\documentclass[11pt]{amsart} 
\usepackage{amsthm,amsbsy,amsfonts,amssymb,amscd}
\usepackage{bbm}
\usepackage[enableskew]{youngtab}
\usepackage[mathscr]{euscript} 

\usepackage{enumerate} 
\usepackage{tikz} 
\usetikzlibrary{calc,intersections,through, backgrounds,arrows,decorations.pathmorphing,backgrounds,positioning,fit,petri} 
\usetikzlibrary{calc}

\usepackage{tikz-cd} 
\usepackage[all]{xy}
\xyoption{knot}
\xyoption{poly}

\usepackage{hyperref}

\DeclareMathAlphabet{\mathpzc}{OT1}{pzc}{m}{it}

\usepackage[margin=1.1in]{geometry}

\newcommand{\SG}{\mathrm{S}}

\newtheoremstyle{notes} {} {} {} {} {\bfseries} {.} {.5em} {}

\theoremstyle{plain}

\newtheorem{prop}[subsubsection]{Proposition}

\newtheorem{lemma}[subsubsection]{Lemma}

\newtheorem{thm}[subsubsection]{Theorem}
\newtheorem{conj}[subsubsection]{Conjecture}

\theoremstyle{remark}

\newtheorem{rem}[subsubsection]{Remark} 
\newtheorem{ddef}[subsubsection]{Definition}

\newcommand{\charr}{\mathrm{char}}
\newcommand{\Dist}{\mathrm{Dist}}

\newcommand{\cc}{\mathbf{c}}

\newcommand{\ev}{\mathrm{ev}}

\newcommand{\tto}{\twoheadrightarrow}
\newcommand{\mZ}{\mathbb{Z}}
\newcommand{\mF}{\mathbb{F}}
\newcommand{\mN}{\mathbb{N}}

\newcommand{\cO}{\mathcal{O}}

\newcommand{\op}{\mathrm{op}}

\newcommand{\OSp}{{\mathsf{OSp}}}
\newcommand{\GL}{{\mathsf{GL}}}
\newcommand{\Pe}{{\mathsf{Pe}}}
\newcommand{\Sp}{{\mathsf{Sp}}}

\newcommand{\OB}{\mathcal{O}\hspace{-0.3mm}\mathcal{B}}

\title[Ringel duality for Brauer algebras]{Ringel duals of Brauer algebras via super groups}

\author{Kevin Coulembier}

\newcommand{\End}{{\rm End}}

\newcommand{\mk}{\Bbbk}
\newcommand{\mC}{\mathbb{C}}
\newcommand{\Hom}{{\rm Hom}}

\newcommand{\Rep}{{\rm Rep}}

\newcommand{\Ext}{{\rm Ext}}
\newcommand{\unit}{{\mathbbm{1}}}
\newcommand{\cC}{{\mathbf{C}}}

\newcommand{\Ob}{{\mathrm{Ob}}}

\newcommand{\oa}{\bar{0}}
\newcommand{\ob}{\bar{1}}

\newcommand{\dd}{\mathbf{d}}

\newcommand{\Grp}{\mathbf{G\hspace{-0.3mm}r\hspace{-0.3mm}p}}

\newcommand{\Aut}{\mathrm{Aut}}

\newcommand{\cA}{\mathcal{A}}

\newcommand{\Sym}{\mathrm{Sym}}

\newcommand{\id}{{\rm id}}
\newcommand{\Frg}{\mathrm{Frg}}

\newcommand{\svec}{\mathbf{svec}}
\newcommand{\csAl}{\mathbf{cs\hspace{-0.3mm}Al}}
\newcommand{\sVec}{\mathbf{s\hspace{-0.4mm}Vec}}
\newcommand{\VVec}{\mathbf{Vec}}
\newcommand{\vvec}{\mathbf{vec}}
\newcommand{\uHom}{\underline{\mathrm{Hom}}}
\newcommand{\uEnd}{\underline{\mathrm{End}}}
\newcommand{\uCEnd}{\underline{\mathrm{C\hspace{-0.2mm}End}}}

\newcommand{\uRep}{\underline{\mathrm{Rep}}}

\newcommand{\usvec}{\underline{\mathbf{svec}}}
\newcommand{\GG}{{\mathsf{G}}}

\newcommand{\Lie}{\mathrm{Lie}}
\newcommand{\mI}{\mathbb{I}}
\newcommand{\sdim}{\mathrm{sdim}}

\newcommand{\cB}{\mathcal{B}}
\newcommand{\cBC}{\mathcal{BC}}
\newcommand{\cOB}{\mathcal{O}\hspace{-0.5mm}\mathcal{B}}
\newcommand{\cOBC}{\mathcal{O}\hspace{-0.5mm}\mathcal{B}\mathcal{C}}

\newcommand{\cR}{\mathcal{R}}
\newcommand{\cH}{\mathcal{H}}
\newcommand{\cS}{\mathcal{S}}

\newcommand{\OG}{\mathsf{O}}
\newcommand{\ad}{\mathrm{ad}}

\newcommand{\JJJ}{\mathscr{J}}

\newcommand{\bA}{{\mathbf{A}}}

\newcommand{\lbr}{[\![}
\newcommand{\rbr}{]\!]}


\subjclass[2010]{}

\begin{document} 
\date{}

\maketitle 



\begin{abstract}
We prove that the Brauer algebra, for all parameters for which it is quasi-hereditary, is Ringel dual to a category of representations of the orthosymplectic super group. As a consequence we obtain new and algebraic proofs
for some results on the fundamental theorems of invariant theory for this super group over the complex numbers and also extend them to some cases in positive characteristic. Our methods also apply to the walled Brauer algebra in which case we obtain a duality with the general linear super group, with similar applications.
\end{abstract}

\section*{Introduction}

We apply Ringel duality to the representation theory of the classical algebraic super groups. Some of the motivation and inspiration for this study originates in the recent construction of the abelian envelope of Deligne's universal monoidal category of \cite{Deligne} in \cite{EHS} via general linear super groups. In the background of the construction in \cite{EHS} one has Ringel duality between finite truncations of the abelian envelope and the original category of Deligne. Recently an infinite version of Ringel duality has been developed rigorously in \cite{BSRingel} which allows to speak of Ringel duality between the actual abelian envelope and the original category of Deligne.

In the current paper, we provide an alternative proof for the truncated duality, for which we do not apply any serious representation theory of super groups. Instead we rely on the theory of Brauer algebras (the endomorphism algebras of objects in Deligne's category) and the classical construction of Schur algebras out of affine group schemes. This also allows us to extend this Ringel duality to positive characteristics and to the orthosymplectic super group, in which cases the results are new. The Ringel dualities we establish allow to transfer information between the representation theory of Brauer algebras and super groups.
The results also have applications to invariant theory of super groups. In particular we re-establish many cases of the first and second fundamental theorem of invariant theory as obtained in e.g.~\cite{Sel, DLZ, ES, Kujawa, BrCat, Yang} in characteristic zero, now with purely algebraic proofs, and extend them to positive characteristic.

The paper is organised as follows. In Sections~\ref{SecGrp} and~\ref{DiagCat} we recall the relevant notions of super algebra and diagram categories. In Section~\ref{SecCOAL} we develop some theory of centraliser coalgebras for representations of small categories. The goal is to prove that the centraliser algebra of the Brauer algebra acting on tensor powers is given by a Schur type algebra coming from the relevant super group. The study of centraliser algebras via coalgebras and Schur algebras is a classical method, see \cite{Green}. However, in order to obtain a rigorous proof which is computation free and applicable to all classical super groups we propose a method which involves the monoidal structure of the relevant Deligne category, rather than relying solely on the Brauer algebra itself. 

In Section~\ref{SecSchur} we apply the theory of Section~\ref{SecCOAL} to the Brauer, the walled Brauer, the periplectic Brauer and the Brauer-Clifford algebra. Furthermore, we identify the abelian subcategory of the category of algebraic representations of the corresponding super group which is described by the centraliser algebra (the Schur algebra).

In Section~\ref{SecRingel} we show that, under certain restrictions on the parameters, the centraliser algebras of the (walled) Brauer algebra as described in Section~\ref{SecSchur} are precisely the Ringel duals of the Brauer algebras. In particular, this establishes a Ringel duality between the (walled) Brauer algebra and a category of representations of a super group, for all parameters for which the Brauer algebra is quasi-hereditary.

In Section~\ref{SecApp} we describe some applications of our results to the representation theory and invariant theory of super groups. In Appendix~\ref{AppRingel} we summarise some elementary properties of quasi-hereditary algebras and Ringel duality which are used in the main part of the paper.

\subsection*{Notation} Throughout the paper, $\mk$ denotes an algebraically closed field of characteristic $p\ge 0$. We will often leave out $\mk$ in subscripts. We let $\mI$ denote the image of $\mZ$ in $\mk$, so $\mI\simeq\mZ$ if $p=0$ and $\mI\simeq\mF_p$ if $p>0$.
We set $\mN=\{0,1,2,\cdots\}$ and denote the cyclic group of order two by $\mZ_2=\mZ/2\mZ=\{\oa,\ob\}$. For $a,b\in\mZ$, we denote by $[\![a,b]\!]$ the set of integers $x$ with $a\le x\le b$.

For $n\in\mN$, we denote the symmetric group on $n$ symbols by $\SG_n$. For $\lambda\vdash n$, we have the corresponding Specht module $S(\lambda)$ of $\mk\SG_n$. We will only consider characteristics of $\mk$ where the Specht modules are simple. For $r\in\mN$, we set 
$$\JJJ(r)\;=\;\{r-2i\,|\, 0\le i\le r/2\}\;\subset\,\mN\quad\mbox{and}\quad\Lambda_r:=\{\lambda\vdash i\,|\, i\in\JJJ(r)\}.$$
For $(r,s)\in\mN\times\mN$, we will also use the following set of bipartitions
$$\Lambda_{r,s}:=\{(\lambda,\mu)\;|\;\lambda\vdash r-j, \,\mu\vdash s-j \;\mbox{for some}\; 0\le j\le \min(r,s)\}.$$

For a finite dimensional associative algebra $A$, we denote its category of finite dimensional left modules by $A$-mod. 
As usual, a subcategory where the morphism sets are identical to those in the original category is called a full subcategory. Dually, a subcategory with same set of objects as the original category will be called a dense subcategory.

\section{Affine super group schemes}
\label{SecGrp}
\subsection{Elementary super algebra}
Super algebra corresponds to algebra in the symmetric monoidal category $\svec_{\mk}$.

\subsubsection{Vector spaces} We denote by $\svec_{\mk}$ the category of finite dimensional $\mZ_2$-graded vector spaces. Morphisms in this category are the grading preserving $\mk$-linear homomorphisms. For super spaces $V$ and $W$ we denoted the space of such morphisms by $\Hom_{\mk}(V,W)$. The category $\svec$ is $\mk$-linear and monoidal. For $v\in V_{i}$, with $i\in\{\oa,\ob\}$ we write $|v|=i$. The parity change functor $\Pi$ satisfies $(\Pi V)_{i}=V_{i+\ob}$.
The monoidal category of all $\mZ_2$-graded vector spaces is denoted by $\sVec_{\mk}$.
We consider $\svec_{\mk}$ and $\sVec_{\mk}$ as symmetric monoidal categories with braiding
$$\gamma_{V,W}: \;V\otimes W\,\stackrel{\sim}{\to}\, W\otimes V,\quad v\otimes w\mapsto (-1)^{|w||v|}w\otimes v.$$
Definitions as the above extend uniquely from homogeneous elements by linearity. By a ``form'' on a super vector space $V$ we will always mean a non-degenerate bilinear symmetric (with respect to $\gamma_{V,V}$) form. Such a form is even, resp. odd, if $\langle v,w\rangle=\oa$ when $|u|+|v|=\ob$, resp. $|u|+|v|=\oa$.

\subsubsection{}For $V\in\svec$, we denote by $\dim V\in\mN\times\mN$ the pair $(m,n)$, with $m$, resp. $n$, the ordinary dimension of $V_{\oa}$, resp. $V_{\ob}$. By $\sdim V\in\mI$ we denote the categorical dimension in $\svec$. Concretely, $\sdim V$ is the image of $m-n$ in $\mk$. When working with such a fixed vector space, we use the function
$$
[\cdot]: \;\lbr1,m+n\rbr \,\to\,\mZ_2,\quad\mbox{with}\quad
[i]\;=\;\begin{cases}
\oa&\mbox{  for $ i\in\lbr1, m\rbr$},\\
\ob&\mbox{ for $ i\in\lbr m+1, m+n\rbr$.}
\end{cases}$$
When we denote a basis of $V$ by $\{e_i\}$, we assume that $e_i\in V_{[i]}$.

The monoidal category $\sVec$ has internal homomorphisms, denoted by
$$\uHom(V,W),\quad\mbox{with} \quad \uHom(V,W)_{\oa}=\Hom(V,W)\quad\mbox{and}\quad \uHom(V,W)_{\ob}=\Hom(V,\Pi W).$$
For $V\in\sVec$, we set $V^\ast=\uHom(V,\mk)$.

\subsubsection{Algebras}
A super algebra $(A,m,\eta)$ is a monoid in $\sVec$. As this does not involve the braiding, this is just a $\mZ_2$-graded algebra. However the definition of the tensor product $A\otimes B$ of two algebras involves the braiding.
Similarly, a commutative super algebra is a commutative monoid in $\sVec$. 
We denote the category of commutative super algebras by $\csAl_{\mk}$.

 By definition, an $A$-module corresponds to an $M\in\sVec$ with a super algebra morphism
$A\to\uEnd(M).$ The superspace $\uHom_A(M,N)$ consists of $f\in \uHom(M,N)$ which satisfy $f(av)=(-1)^{|a||f|}af(v)$, for homogeneous $a\in A$ and $v\in M$. It is easy to see that we get an isomorphic super space when we do not impose a minus sign in the commutation relation.

\subsubsection{Categories}\label{SupCat} Super categories and super functors are generalisations of super algebras and super modules in the same way as $\mk$-linear categories and functors are generalisations of algebras and modules. Super categories and super functors are thus enriched over the monoidal category $\sVec$. We will often interpret a super category with finitely many objects simply as a super algebra (with some distinguished idempotents).

A monoidal super category is defined similarly as a $\mk$-linear monoidal category, but based on the super interchange law
$$(f\otimes g)\circ (h\otimes k)\;=\; (-1)^{|g||h|} (f\circ h\otimes g\circ k).$$
We refer to \cite{BE} for a complete treatment of monoidal super categories.
 An example of a monoidal supercategory is the category $\usvec$ which has the same objects as $\svec$, but morphism superspaces given by the internal morphism spaces. 
 As a manifestation of the super interchange law, $f\otimes g$ for two homogeneous morphisms $f,g$ in $\usvec$ has to be interpreted as
 $$(f\otimes g)(v\otimes w)\;=\; (-1)^{|g||v|} f(v)\otimes g(w),$$ 
 with $v,w$ homogeneous elements in the relevant super spaces.

\subsubsection{Coalgebras} A super coalgebra $(C,\Delta,\varepsilon)$ is a comonoid in $\sVec$, see e.g. \cite[\S 2.1.1]{Abe}. We denote by scom-$C$, the category of comodules in $\svec$. By definition, a comodule is a finite dimensional super vector space $M$ with a morphism
$$c_M:\, M\to M\otimes C \qquad\mbox{in $\sVec$},$$
  such that
  $$(\id_M\otimes \varepsilon)\circ c_M=\id_M\qquad\mbox{and}\qquad (\id_M\otimes \Delta)\circ c_M=(c_M\otimes \id_C)\circ c_M.$$
  
  When we consider the category of comodules in $\svec$ of $C$, but with all (not necessarily grading preserving) morphisms which commute with the coaction, we write $\underline{{\rm scom}}$-C. The category $\underline{{\rm scom}}$-$C$ is a super category, but not necessarily abelian. Denote by com-$C$ the category of finite dimensional comodules of $C$, regarded as an ordinary coalgebra. By definition, we have forgetful functors
\begin{equation}\label{eqFrg}\Frg:\;\mbox{scom-}C\;\to\;\mbox{com-}C\quad\mbox{and}\quad \underline{\Frg}:\;\underline{{\rm scom}}\mbox{-}C\;\to\;\mbox{com-}C,
\end{equation} where the latter is fully faithful.

\subsubsection{Matrix coefficients}\label{defCM} Fix a coalgebra $C$ in $\sVec$. For $M$ in scom-$C$, we have a coalgebra morphism
\begin{equation}\label{eqMC}[-|-]_C:\;M^\ast\otimes M\to C,\quad \alpha\otimes v\mapsto  [\alpha|v]_C\;:=\; (\alpha\otimes \id_C)\circ c_M(v). \end{equation}
We denote by $C^M$ the image of the above morphism.
By definition, $C^M$ is a subcoalgebra of $C$ and $M$ restricts to a comodule of $C^M$. 
  For a short exact sequence in scom-$C$  $$0\to M_1\to M\to M_2\to 0,$$
  we have $C^M\supset C^{M_1}+C^{M_2}$, $C^{M^{\oplus n}}=C^M$ and $C^{\Pi M}=C^M$.

For any finite dimensional subcoalgebra $B$ of $C$, we can therefore consider scom-$B$ canonically as the abelian subcategory of scom-$C$ of all comodules which are subcomodules of direct sums of the $C$-comodules $B$ and $\Pi B$.

\subsubsection{}\label{dualA} For a coalgebra $(C,\Delta,\varepsilon)$ in $\sVec$, we have the dual algebra $(C^\ast,m,\eta)$, with
$$m(\alpha,\beta)(x):=(\alpha\otimes\beta)\circ\Delta(x), \quad\mbox{for $\alpha,\beta\in C^\ast$ and $x\in C$, and } \eta(\lambda):=\lambda\varepsilon.$$
We could also define the dual with $m$ twisted by $\gamma$ in $\sVec$. Then taking the dual would intertwine the functors forgetting the $\mZ_2$-grading, but would interchange comodule structures on a space $V$ with module structures on $V^\ast$, rather than on $V$.


\subsubsection{Hopf algebras} 
A super bialgebra is a tuple $(A,m,\eta,\Delta,\varepsilon)$ in $\sVec$ such that $(A,m,\eta)$ is a monoid and $(C,\Delta,\varepsilon)$ is a comonoid for which $\Delta$ and $\varepsilon$ are algebra morphisms (in $\sVec$), see \cite[Theorem~2.1.1]{Abe}. 
If we additionally have an antipode $S$, see \cite[\S 2.1.2]{Abe}, then 
 $(A,m,\eta,\Delta,\varepsilon,S)$ is a Hopf super algebra. The antipode $S$ is unique if it exists.
 We say that a Hopf super algebra $A$ is commutative if $(A,m,\eta)$ is commutative in $\sVec$.

For a Hopf super algebra $A$ and two comodules $M,N$ the comodule structure on $M\otimes_{\mk}N$ is given by
\begin{equation}\label{eqTP}M\otimes N\stackrel{c_M\otimes c_N}{\to}M\otimes C\otimes N\otimes C\stackrel{M\otimes \gamma_{C,N}}{\to}M\otimes N\otimes C\otimes C\stackrel{M\otimes N\otimes m}{\to}M\otimes N\otimes C.\end{equation}

The comodule structure on the dual space of a comodule $M$ is defined by
$$c_{M^\ast}(e_i^\ast)=\sum_j (-1)^{[j]([i]+[j])}e^\ast_j\otimes S(c_{ij}),$$
with $\{e_i\}$ a basis of $M$ with $c_M(e_j)=\sum_ie_i\otimes c_{ij}$.

\subsection{Algebraic super groups}
By the Yoneda embedding, we have an equivalence between the category of commutative Hopf super algebras and the opposite of the category of representable functors
$$\csAl_{\mk}\to \Grp.$$
The latter are known as affine super group schemes, or simply algebraic super groups in case the corresponding Hopf super algebra is finitely generated as an algebra. 

Now fix a finitely generated commutative Hopf super algebra $\cO$ and set
 $\GG={\csAl}(\cO,-)$ and
$$\Rep\GG\;:=\;\mbox{scom-}\cO\quad\mbox{and}\quad \uRep\GG\;:=\;\underline{{\rm scom}}\mbox{-}\cO.$$

\subsubsection{} We set $\cO_+=\ker(\varepsilon)$, with $\varepsilon$ the counit of $\cO$, and
$$\Dist(\cO):= \bigcup_{i>0}(\cO/(\cO_+)^i)^\ast\;\subset\;\cO^\ast.$$
Then $\Dist(\cO)$ inherits a Hopf algebra structure from $\cO$, see \cite[\S 2.3.5]{Abe}, as an extension of the procedure in \ref{dualA}. 
We define $\Lie(\GG)$ as the space of primitive elements in $\Dist(\cO)$: 
$$\Lie(\GG)\;:=\{d\in \cO^\ast\,|\, d(fg)= d(f)\varepsilon(g)+\varepsilon(f)d(g),\;\mbox{for all $f,g\in\cO$}\}.$$
We interpret $\Lie(G)$ as a Lie super algebra for multiplication given by the supercommutator in the associative algebra $\Dist(\cO)$.

\begin{lemma}\label{LemDist}\cite[Remark~2(2) and Lemma~19]{Masuoka}
\begin{enumerate}[(i)]
\item If $\charr(\mk)=0$, we have $\Dist(\cO)=U(\Lie \GG)$.
\item If $\GG_{ev}$ is connected, the natural pairing $\Dist(\cO)\times\cO\to\mk$ satisfies
$$u(f)\not=0\;\mbox{for some $u\in \Dist(\cO)$ for all $0\not=f\in\cO$.}$$
\end{enumerate}
\end{lemma}

\subsubsection{}\label{Omod} Fix a $(M,c_M)\in \mbox{scom-}\cO$. 
This is naturally a $\Dist(\cO)$-module, for
$$\Dist(\cO)\to \uEnd(M),\quad u\mapsto  (id_M\otimes u)\circ c_M,$$
and thus by restriction also a $\Lie\GG$-module.

Now consider the finite dimensional coalgebra $\cO^M$ as in \ref{defCM}. 
We have the finite dimensional algebra $S^M:=(\cO^M)^\ast$ and an injective algebra morphism
$$S^M\hookrightarrow\uEnd(M),\quad \alpha\mapsto (\id_M\otimes \alpha)\circ c_M. $$
By construction, the above maps yield a commutative diagram of algebra morphisms
\begin{equation}\label{CommD}\xymatrix{
\cO^\ast\ar@{->>}[r]&S^M\ar@{^{(}->}[r]& \uEnd(M).\\
&\Dist(\cO)\ar[ur]\ar@{^{(}->}[ul]
}\end{equation}

\subsubsection{} For each $g\in \GG(\mk)$, we have the algebra morphism
$$\ad_g:=(g\otimes \id\otimes g)\circ (\Delta\otimes S)\circ \Delta\;:\;\; \cO\to\cO.$$
The following lemma, in which we ignore monoidal structures, is standard for the special case $g^2=\varepsilon$ ({\it i.e.} a homomorphism $\mZ_2\to\GG$).

\begin{lemma}\label{GenZ2}
Assume $p\not=2$ and there exists $g\in \GG(\mk)$ such that $\ad_g(f)=(-1)^{|f|}f$, for each homogeneous $f\in\cO$. Then $\Rep\GG$ has full subcategories $\cC_1$ and $\cC_2$ yielding equivalences
$$\cC_1\oplus\cC_2\stackrel{\sim}{\to} \Rep\GG\quad\mbox{and}\quad \Pi: \cC_1\stackrel{\sim}{\to}\cC_2.$$ 
Furthermore, with $i\in{1,2}$ and $M\in\Rep\GG$, the functor $\Frg$ in \eqref{eqFrg} restricts to equivalences
$$\cC_i\stackrel{\sim}{\to} \mbox{{\rm com-}}\cO\quad\mbox{and}\quad ( \cC_i\cap\mbox{{\rm scom-}}\cO^M) \stackrel{\sim}{\to} \mbox{{\rm com-}}\cO^{\Frg M}.$$
\end{lemma}
\begin{proof}
We start by considering $\cO$ as an ungraded coalgebra. For every $M\in \mbox{com-}\cO$ we have
$$a_M:= (\id\otimes g)\circ c_M\;\in\;\End_{\mk}(M)$$
with commutation relations
\begin{equation}
\label{commaM}(a_M\otimes\id)\circ c_M\;=\; (\id\otimes \ad_g)\circ c_M\circ a_M\qquad\mbox{and}\qquad f\circ a_M=a_N\circ f,
\end{equation}
for all $f\in \Hom_{\cO}(M,N)$.

Take an indecomposable $M\in\mbox{com-}\cO$. It follows by assumption and the first equation in \eqref{commaM} that the vector space $M$ decomposes into two generalised eigenspaces of $a_M$ with eigenvalues $\lambda$ and $-\lambda$ for some $\lambda\in\mk^{\times}$. Furthermore, we can impose precisely two $\mZ_2$-gradings on $M$ which are compatible with $c_M$ for $\cO$ now regarded as a $\mZ_2$-graded coalgebra. Concretely we can set $M_{\oa},M_{\ob}$ equal to the generalised eigenspaces of $a_M$. 

For any indecomposable $M\in$ com-$\cO$, we choose one of the two options as $\widetilde{M}\in$ scom-$\cO$ with $\Frg\widetilde{M}=M$. We define $\cC_1$ as the full subcategory of direct sums of comodules $\widetilde{M}$ and $\cC_2$ as the full subcategory of direct sums of comodules $\Pi\widetilde{M}$. It follows from the second equation in \eqref{commaM} that we have
$$\Hom_{\cO}(\widetilde{M},\Pi\widetilde{N})=0\quad\mbox{or}\quad\uHom_{\cO}(\widetilde{M},\widetilde{N})=\Hom_{\cO}(\widetilde{M},\widetilde{N}).$$
From this observation, all claims in the lemma now follow.
\end{proof}

\subsection{The classical algebraic super groups}

\subsubsection{The general linear super group}\label{DefGL} Fix $V\in \svec$, set $(m,n):=\sdim V$ and for each $R\in\csAl$ we consider the right $R$-module $V_R:=V\otimes_{\mk}R$. 
The functor
$$\GL(V):\csAl\to \Grp,\quad R\mapsto \Aut_R(V_R)$$
is represented by $\cO[\GL(V)]$. Consider variables $X_{ij}$ and $Z_{ij}$ for $1\le i,j\le m+n$ of parity $|X_{ij}|=[i]+[j]=|Z_{ij}|$. After choosing a basis of $V$ we can describe $\cO[\GL(V)]$ as the quotient of a polynomial super algebra as
$$\cO[\GL(V)]\;:=\; \mk[X_{ij}, Z_{kl}]/I,\quad\mbox{with}\quad I:=\langle \sum_jZ_{ij}X_{jk}=\delta_{ik} \rangle.$$
Furthermore, we define
$$\Delta(X_{ij})\;=\;\sum_{k}X_{ik}\otimes X_{kj},\quad\varepsilon(X_{ij})=\delta_{ij},\quad\Delta(Z_{jk})=\sum_{l}(-1)^{([j]+[l])([l]+[k])}Z_{lk}\otimes Z_{jl}$$
and $S(X_{ij})=Z_{ij}$. 
We have $g\in\csAl(\cO,\mk)$ defined by $g(X_{ij})=(-1)^{[i]}\delta_{ij}$ which satisfies the condition in Lemma~\ref{GenZ2}.

Now $V$ is the natural $\GL(V)$-representation, with dual $V^\ast$, determined by the coactions
$$e_i\mapsto \sum_{j}e_j\otimes X_{ji}\quad\mbox{and}\quad e_i^\ast\mapsto \sum_{j}e_j^\ast\otimes Y_{ji},$$
with $Y_{ji}:=(-1)^{[j]([i]+[j])}Z_{ij}$. In particular, we have $[e_j^\ast|e_j]_{\cO}=X_{ji}$.
We use the choice of simple positive roots of \cite[Section~4.4]{EHS}. In particular the highest weight of $V$ is $\epsilon_1$ and of $V^\ast$ it is $-\delta_1$. 

\subsubsection{The orthosymplectic super group} \label{DefOSp}
Consider $V\in \svec$ with an even form $\langle\cdot,\cdot\rangle: V\times V\to \mk$, which we also interpret in $\End(V^{\otimes 2}, \mk)$. Then $\OSp(V)$ is the subgroup of $\GL(V)$ which preserves this form.
Concretely, if we choose a basis $\{e_i\}$ of $V$ and set $g_{ij}=\langle e_i,e_j\rangle$ then $\cO[\OSp(V)]$ is the quotient of $\cO[\GL(V)]$ with respect to the ideal generated by the elements
\begin{equation}\label{eqform}\sum_{k,l}(-1)^{[l]([k]+[i])}X_{ki}g_{kl}X_{lj}\;-\; g_{ij}.\end{equation}
We have $g\in\csAl(\cO[\OSp(V)],\mk)$ defined by $g(X_{ij})=(-1)^{[i]}\delta_{ij}$, which satisfies the condition in Lemma~\ref{GenZ2}.

We refer to \cite{SW} for an introduction to the algebraic representation theory of $\GG=\OSp(V)$. In particular we follow the conventions {\it loc. cit.} and denote the simple highest weight module, see \cite[Lemma~4.1]{SW}, with highest weight $\xi\in X^+$ by $L_{\GG}(\xi)$. For the special modules we will encounter we do not need the full description of $X^+$. We will only need weights of the form $\xi=\sum_{i=1}^n\lambda_i\delta_i$, for partitions $\lambda$, which are to be interpreted as dominant weights for $\Sp(V_{\ob})$.

\subsubsection{The periplectic super group}
Consider $V\in \svec$ with an odd form $\langle\cdot,\cdot\rangle: V\times V\to \mk$, which we also interpret in $\uEnd(V^{\otimes 2}, \mk)_{\ob}$. This implies that $\dim V_{\oa}=\dim V_{\ob}$, so in particular $\sdim V=0$. Then $\Pe(V)$ is the subgroup of $\GL(V)$ which preserves this form.
Concretely, if we choose a basis $\{e_i\}$ of $V$ and set $g_{ij}=\langle e_i,e_j\rangle$ then $\cO[\Pe(V)]$ is the quotient of $\cO[\GL(V)]$ with respect to the ideal generated by the elements \eqref{eqform}.
Take $\imath\in\mk$ with $\imath^2=-1$, then $g\in\csAl(\cO[\Pe(V)],\mk)$ defined by $g(X_{ij})=(-1)^{[i]}\iota\delta_{ij}$ satisfies the condition in Lemma~\ref{GenZ2}.

\subsubsection{The queer super group}
Consider $V\in \svec$ with $q\in \uEnd(V)_{\ob}$ for which $q^2=\id$. This implies that $\dim V_{\oa}=\dim V_{\ob}$. Then $\mathsf{Q}(V)$ is the subgroup of $\GL(V)$ which commutes with $q$. 

Concretely, if we set $q(e_i)=\sum_j q_{ij}e_j$, the Hopf algebra $\cO[\mathsf{Q}(V)]$ is the quotient of $\cO[\GL(V)]$ with respect to the ideal generated by
$$X_{ij}\;-\;\sum_{k} X_{ik} q_{jk}.$$


\section{Diagram categories}\label{DiagCat}We briefly review some diagram categories. Since we will only use them rather superficially, we do not present full details here.

\subsection{The Brauer category}
\subsubsection{}\label{BrCat1} For $\delta\in\mk$, the $\mk$-linear Brauer category $\cB(\delta)$ is introduced in \cite[\S 2.1]{BrCat}, see also~\cite[\S 9]{Deligne}. Concretely, the objects in $\cB(\delta)$ are given by
$$\Ob\cB(\delta)\;=\;\{[i]\,|\, i\in \mN\}$$
and the space of morphisms from $[i]$ to $[k]$ is given by the $\mk$-span of all pairings of $i+k$ dots. Such a pairing is graphically represented by an $(i,k)$-Brauer diagram, which is a diagram where $i+k$ points  are placed on two parallel horizontal lines, $i$ on the lower line and $k$ on the upper, with
arcs drawn to join points which are paired. Arcs connecting two points on the lower, resp. upper, line are caps, resp. cups. Composition of morphisms corresponds to concatenation of diagrams with loops evaluated at $\delta$. The Brauer category is monoidal with $[i]\otimes [j]=[i+j]$. In \cite[\S 2.2]{BrCat}, a contravariant auto-equivalence ${}^\ast$ of $\cB(\delta)$ is introduced, which is the identity on objects and maps a diagram to its reflection in a horizontal line.

\subsubsection{}\label{DefBHR} We define dense subcategories 
$$\cB(\delta) \;\supset\;\cR\;\supset\; \cH\,\simeq\,\bigoplus_{i\in\mN}\mk\SG_i,$$
where the morphism spaces in $\cH$ are spanned by all diagrams without cups or caps and in $\cR$ they are spanned by all diagrams without cups.
As the notation suggests, $\cR$ and $\cH$ do not depend on the parameter $\delta$.

We define some unital associative algebras for $r\in\mN$. We have the set of objects $[\JJJ(r)]=\{[i]\,|\, i\in\JJJ(r)\}$. The algebra $\cB_r(\delta)$, resp. $\cB_r^{\cc}(\delta)$, is the full subcategory of $\cB(\delta)$ with objects $[r]$, resp. $[\JJJ(r)]$. The algebra $\cR_r$, resp. $\cH_r$, is the full subcategory of $\cR$, resp. $\cH$, with objects $[\JJJ(r)]$.
We can interpret modules over $\cH_r$ as $\cR_r$-modules where every diagram with a cap acts trivially.

\begin{prop}\label{PropCZ}
Fix $r\in\mN$ and $\delta\in\mk$.
\begin{enumerate}[(i)]
\item If $\delta\not\in\mI$, then $\cB_r(\delta)$ is semisimple. If $\delta=0$ and $r$ is even, or if $p\in\lbr 2,r\rbr$, then $\cB_r(\delta)$ is not quasi-hereditary.
\item If $\delta\not=0$ or $r$ is odd, the algebras $\cB^{\cc}_r(\delta)$ and $\cB_r(\delta)$ are Morita equivalent.
\item The right $\cR_r$-module $\cB_r^{\cc}(\delta)$ is projective.
\item If $p\not\in\lbr 2,r\rbr$, the simple $\cB_r^{\cc}(\delta)$-modules can be labelled by $\Lambda_r$. Furthermore, $(\cB_r^{\cc}(\delta),\le)$ is quasi-hereditary for $\lambda<\mu$ if and only if $|\mu|<|\lambda|$ and with standard modules $\Delta(\lambda):=\cB^{\cc}_r(\delta)\otimes_{\cR_r}S(\lambda)$.
\item The restriction of $\ast$ in \ref{BrCat1} to an anti-automorphism of $\cB_r^{\cc}(\delta)$ is a good duality (as in \ref{DefGood}) of the quasi-hereditary algebra $(\cB_r^{\cc},\le)$, if $p\not\in\lbr2,r\rbr$.
\end{enumerate}
\end{prop}
\begin{proof}The first statement in part (i) follows from the main result of \cite{Rui}. The second statement in part (ii) is \cite[Theorem~1.3]{CellQua}.
Part (ii) is \cite[Theorem~8.5.1]{Borelic}. Part (iii) is proved in \cite[Proposition~8.4.4]{Borelic}, by \cite[Definition~3.2.3]{Borelic}.  Part (iv) is \cite[Theorem~8.4.1]{Borelic}. 
Since the equivalence $\ast$ is the identity on $\Ob\cB(\delta)$, the duality of $\cB_r^{\cc}(\delta)$ preserves the partial order of part (iv), proving part (v).
\end{proof}

\subsubsection{}
For $V\in\svec$ with even form and $\delta:=\sdim(V)$, we have a $\mk$-linear symmetric monoidal functor 
\begin{equation}\label{UnivO}
\cB(\delta)\;\to\;\Rep\OSp(V),\qquad\mbox{with}\quad [1]\mapsto V\quad\mbox{and}\quad \cap\mapsto \left(\langle\cdot,\cdot\rangle : V^{\otimes 2}\to \mk \right).
\end{equation}
This is well-known and follows from a straightforward extension of \cite[Theorem~3.4]{BrCat}.

\subsection{Oriented Brauer category}
\subsubsection{} For $\delta\in\mk$, we have the oriented Brauer category $\cOB(\delta)$, with objects given by finite words in the alphabet $\{\vee,\wedge\}$. For a complete definition we refer to e.g.~\cite[\S 4]{ES}. The walled Brauer algebra $\cB_{r,s}(\delta)$ is the full subcategory of $\OB(\delta)$ corresponding to the object $\vee^{\otimes r}\otimes \wedge^{\otimes s}$ and $\cB_{r,s}^{\cc}(\delta)$ is the full subcategory of $\OB(\delta)$ corresponding to the objects $\{\vee^{\otimes r-i}\otimes \wedge^{\otimes s-i}\}$ for $i\in\lbr0,\min(r,s)\rbr$.
The analogues of (ii)-(v) in Proposition~\ref{PropCZ} are proved in \cite[\S 8]{Borelic}. 

\subsubsection{} Fix $V\in\svec$ and set $W:=V^\ast$ with pairing 
$\ev_V:W\otimes V\to\mk$ given by $\alpha\otimes v\mapsto \alpha(v).$
If $\delta=\sdim(V)$, we have a $\mk$-linear symmetric monoidal functor 
\begin{equation}\label{UnivG}
\cOB(\delta)\;\to\;\Rep\GL(V),\qquad\mbox{with}\quad \vee\mapsto V\;\;\mbox{and}\;\; \wedge\mapsto W,
\end{equation}
where the unique oriented Brauer diagram which represents a morphism from $\wedge\vee$ to the empty word is mapped to $\ev_V$.

\subsection{Periplectic Brauer category} 
\subsubsection{}In \cite{Kujawa}, the periplectic Brauer supercategory $\cA$ is introduced. This category has the same set of objects and spaces of morphisms as $\cB(\delta)$. The composition of morphisms in $\cA$ is again given by concatenation of diagrams, up to possible minus signs, with evaluation of loops at $0$. This is a monoidal super category, see also~\cite{BE, PB1}. The periplectic Brauer algebra $\cA_r$ is the full subcategory of $\cA$ corresponding to the object $[r]$. This is actually a reduced super algebra, {i.e.} $(\cA_r)_{\ob}=0$. We also consider the full subcategory $\cA^{\cc}_r$ of $\cA$ corresponding to the objects $[\JJJ(r)]$.

\subsubsection{} Fix $V\in\svec$ with odd form $\langle\cdot,\cdot\rangle$. We have a $\mk$-linear symmetric monoidal super functor 
\begin{equation}\label{UnivP}
\cA\;\to\;\uRep\Pe(V),\qquad\mbox{with}\quad [1]\mapsto V\quad\mbox{and}\quad \cap\mapsto \left(\langle\cdot,\cdot\rangle : V^{\otimes 2}\to \mk \right),
\end{equation}
see \cite[Theorem~5.2.1]{Kujawa}.

\subsection{Oriented Brauer-Clifford category} \label{SecOBC}
\subsubsection{}In \cite{ComesKuj}, the oriented Brauer-Clifford supercategory $\cOBC$ is introduced. It contains $\cOB(0)$ as a dense subcategory. But also has an odd isomorphism $\widetilde{q}$ of $\vee$. The Brauer-Clifford algebra $\cBC_{r,s}$ is the full subcategory of $\cOBC$ corresponding to the object $\vee^{\otimes r}\otimes \wedge^{\otimes s}$.

\subsubsection{} Fix $V\in\svec$ with an odd endomorphism $q$ with $q^2=\id$. By \cite[\S 4.2]{ComesKuj}, we have a monoidal super functor from $\cOBC$ to $\uRep\mathsf{Q}(V)$ which yields a commuting diagram 
$$\xymatrix{
\cOBC\ar[rr]&& \uRep \mathsf{Q}(V)\\
\cOB(0)\ar[rr]\ar[u]&& \Rep\GL(V)\ar[u],
}$$
where the lower horizontal arrow is the functor in \eqref{UnivG}, the left vertical arrow is the inclusion and the right vertical arrow is a forgetful functor. Furthermore, $\widetilde{q}$ is mapped to $q$.

\section{Centraliser coalgebras and monoidal functors}\label{SecCOAL}

\subsection{Definitions and basic properties}

\subsubsection{}\label{coend}Fix a small super category $\bA$ with super functor $F:\bA\to \usvec$. 
We define a super vector space
$$\uCEnd_{\bA}(F)\;:=\;\left(\bigoplus_{X\in\Ob\bA}F(X)^\ast\otimes F(X)\right)/I$$
with $I$ the space
$$I:=\{\alpha\circ F(f)\otimes v-\alpha\otimes F(f)(v)\,|\, \alpha\in F(Y)^\ast, v\in F(X), f\in \bA(X,Y)\;\mbox{ and } X,Y\in\Ob\bA\}.$$

\begin{ddef}
The centraliser coalgebra of the functor $F$ is the superspace $\uCEnd_{\bA}(F)$
with structure morphisms 
$$\varepsilon:\uCEnd_{\bA}(F)\to\mk,\;\; \alpha\otimes v\mapsto \alpha(v)\qquad\mbox{and}$$
$$ \Delta: \uCEnd_{\bA}(F)\to \uCEnd_{\bA}(F)\otimes_{\mk} \uCEnd_{\bA}(F),\;\; \alpha\otimes v\mapsto\sum_i (\alpha\otimes e_{i})\otimes (e_{i}^\ast\otimes v),$$
for $\alpha\in F(X)^\ast$ and $v\in F(X)$ where $\{e_{i}\}$ denotes a basis of $F(X)$, with $X\in \Ob\bA$.
\end{ddef}

\begin{rem}
\begin{enumerate}[(i)]
\item The coalgebra $\uCEnd_{\bA}(F)$ would be the same (after forgetting the grading) if we interpret $F$ only as a $\mk$-linear functor. 
\item If we interpret $F$ as the module $\oplus_{X\in\Ob\bA}F(X)$ for the super algebra $\bA$, then $\uCEnd_{\bA}(F)$ is just the coalgebra $F^\ast\otimes_{\bA}F$.

\item If the $\bA$-module $F$ is finite dimensional, then $\uCEnd_{\bA}(F)$ is $\uEnd_{\bA}(F)^\ast$.
\end{enumerate}
\end{rem}

The following lemma, see also~\cite[Lemma~2.7]{BSRingel}, states that the dual algebra, as in \ref{dualA}, of the centraliser coalgebra is the ordinary centraliser algebra, justifying the name of the former.
\begin{lemma}\label{LemAC}
We have a super algebra isomorphism $(\uCEnd_{\bA}(F))^\ast\simeq \uEnd_{\bA}(F)$.
\end{lemma}
\begin{proof}
By tensor-hom adjunction, we have a canonical isomorphism
$$\Hom_{\bA}(F,F^{\ast\ast})\;\stackrel{\sim}{\to}\; \uHom_{\mk}(F^{\ast}\otimes_{\bA} F,\mk).$$
It is clear that every $\bA$-linear morphism from $F$ to $F^{\ast\ast}$ factors through $F\subset F^{\ast\ast}$.
It then follows by direct computation that this exchanges the algebra structures.
\end{proof}

\subsection{Monoidal structures}
\subsubsection{} Fix a strict monoidal super category $(\bA,\otimes,\unit_{\bA})$ and a strict monoidal super functor $F:\bA\to\usvec$. 
The identity $\mk=F(\unit_{\bA})$ allows to define the morphism $\eta$ in $\sVec$ as composition
$$\eta: \mk\stackrel{\sim}{\to}(F(\unit_{\bA}))^\ast\otimes F(\unit_{\bA})\hookrightarrow C^0_F.$$
For all $X,Y\in\Ob\bA$, the identity $F(X)\otimes_{\mk}F(Y)=F(X\otimes Y)$ allows to define
$$m_{X,Y}:\; \uCEnd_{\mk}(F(X))\otimes_{\mk} \uCEnd_{\mk}(F(Y))\;\stackrel{\sim}{\to }\;  \uCEnd_{\mk}(F(X\otimes Y))$$
$$ (\alpha\otimes v)\otimes (\beta\otimes w)\mapsto (-1)^{|v||\beta|} (\alpha\otimes \beta)\otimes (v\otimes w).$$
The latter morphisms together yield a morphism $m:C^0_F\otimes_{\mk}C^0_F\to C^0_F$.
It follows from the definitions of monoidal super functors that $(C^0_F,m,\eta,\Delta,\varepsilon)$ is a bialgebra in $\sVec$.

\begin{lemma}\label{LemBIA}
Consider an affine super group scheme $\GG$ and a monoidal super category $\bA$ with monoidal super functor $\bA\to\uRep \GG$. Denote by $F$ the composition of this functor with the forgetful functor $\uRep \GG\to\usvec$. Then we have a super bialgebra morphism
$$\phi:\;\uCEnd_{\bA}(F)\;\to\; \cO[\GG],\qquad \alpha\otimes v\mapsto [\alpha|v]_{\cO[\GG]}.$$
\end{lemma}
\begin{proof}
By equation~\eqref{eqMC} we have a super coalgebra morphism $\oplus_{X}(F(X)^\ast\otimes F(X))\to \cO$.
That the morphism is zero on the space $I$ of Subsection~\ref{coend} follows from the fact that $F(f)$ is an $\cO$-comodule morphism. That $\phi$ is an algebra morphism is a direct consequence of equation~\eqref{eqTP}.
\end{proof}

\begin{lemma}\label{LemNew}
Keep the assumptions of Lemma~\ref{LemBIA} and assume that $\phi$ is an isomorphism. For a finite set $E\subset \Ob\bA$ the super space $M:=\oplus_{X\in E}F(X)$ is naturally an object in $\uRep\GG$ and an $A$-module for the super algebra $A:=\oplus_{X,Y\in E}\bA(X,Y)$. The coalgebra morphism $M^\ast\otimes M\to\cO$ of \eqref{eqMC} factors through an isomorphism
$$M^\ast\otimes_AM\;\stackrel{\sim}{\to}\; \cO^M.$$
\end{lemma}
\begin{proof}
By construction, $M^\ast\otimes_AM$ is the image of the canonical morphism $M^\ast\otimes M\to \uCEnd_{\bA}(F)$. Also from construction follows that the composition
$$M^\ast\otimes M\to  \uCEnd_{\bA}(F)\stackrel{\phi}{\to}\cO[\GG] $$
is equal to morphism~\eqref{eqMC} which by definition has image $\cO^M$.
\end{proof}

\subsection{Applications}

\begin{thm}\label{ThmNew}
For the four super monoidal functors $\bA\to \uRep\GG$ of Section~\ref{DiagCat}, the bialgebra morphism $\phi$ of Lemma~\ref{LemBIA} is an isomorphism.
\end{thm}
\begin{proof}
Since the algebras $\cO[\GG]$ are finitely presented the claim can be easily verified. As an example we treat the case $\cB(\delta)\to\uRep\OSp(V)$ of \eqref{UnivO}.

Since every object in $\cB(\delta)$ is a tensor power of $[1]$ it follows that the algebra $\uCEnd_{\cB(\delta)}(F)$ is generated by $V^\ast\otimes V=F([1])^\ast\otimes F([1])$. By definition, $\phi$ maps the latter space to the $\mk$-span of the generators $\{X_{ij}\}$ of $\cO[\OSp(V)]$. This already implies that $\phi$ is surjective, and to complete the proof it suffices to show that the relations in $\cO[\OSp(V)]$ between the generators $X_{ij}$ are elements in the space $I$ of \ref{coend} defining $\uCEnd_{\cB(\delta)}(F)$. The commutation relations $X_{ij}X_{kl}=(-1)^{([i]+[j])([k]+[l])}X_{kl}X_{ij}$ correspond to the elements in $I$ induced by braiding endomorphism of $[1]\otimes [1]$ in $\cB(\delta)$. The relations \eqref{eqform} correspond to the elements in $I$ induced by $\cap$ in $\cB(\delta)$
\end{proof}

\begin{rem}
It is easy to see that $\phi$ is always surjective when the image of $\bA\to\uRep\GG$ contains a tensor generator of $\Rep \GG$ (and its dual). It also follows in general that $\phi$ is injective when $\bA\to\uRep\GG$ is full. The latter is not a necessary condition however. For instance, the super functor $\cR\to\uRep\OSp(V)$, with $\cR$ the dense subcategory of $\cB(\delta)$ of \ref{DefBHR}, is not full but leads to an isomorphism $\phi$.
\end{rem}


\section{Super Schur algebras}\label{SecSchur}

\subsection{The orthosymplectic case}\label{SchurO}
We a fix $V\in\svec$ with an even form. We set $(m|2n):=\dim V$ and $\delta:=m-2n=\sdim(V)\in\mI$. 

\subsubsection{} Now we fix $r\in\mN$ and we set
$$T^r\;:=\;\bigoplus_{j\in\JJJ(r)}V^{\otimes j}.$$ Composition of the functor in \eqref{UnivO} with the forgetful functor to $\svec$ yields algebra morphisms 
$$\cB^{\cc}_r(\delta)\to \End_{\mk}(T^r)\quad\mbox{and}\quad \cB_r(\delta)\to\End_{\mk}(V^{\otimes r}).$$ We define the super algebra
$$\cS^o_r(V)\;:=\;\uEnd_{\cB_r}(V^{\otimes r}).$$

\begin{thm}\label{ThmOSp}
Set $\GG=\OSp(V)$.
\begin{enumerate}[(i)]
\item We have $\cS^o_r(V)\simeq \uEnd_{\cB^{\cc}_r}(T^r)$.
\item If $p\not=2$, the category $\cS^o_r(V)\mbox{{\rm -mod}}$ is equivalent to the abelian subcategory $\Rep^{(r)}\GG$ of modules in $\Rep\GG$ which are subquotients of direct sums of $V^{\otimes r}$.
\item If $n\ge r$ and $p\not\in\lbr2,r\rbr$, the simple module $L_{\GG}(\xi)$ for $\xi\in X^+$ is contained in $\Rep^{(r)}\GG$ if and only if $\xi=\sum_{i=1}^r \lambda_i\delta_i$ for some $\lambda\in \Lambda_r$.
\end{enumerate}
\end{thm}

\begin{lemma}\label{LemOOSp}
For $\cO:=\cO[\OSp(V)]$, we have algebra isomorphisms $\cS^o_r(V)\simeq (\cO^{V^{\otimes r}})^\ast$  and $\uEnd_{\cB^{\cc}_r}(T^r)\simeq(\cO^{T^r})^\ast$.
\end{lemma}
\begin{proof}
These are applications of Lemma~\ref{LemNew}, by Theorem~\ref{ThmNew}.
\end{proof}

\begin{lemma}\label{numbsimp}
If $n\ge r$, for each $\xi=\sum_{i=1}^r \lambda_i\delta_i$  with $\lambda\in \Lambda_r$, the simple module $L_{\GG}(\xi)$ is contained in $\Rep^{(r)}\GG$.
\end{lemma}
\begin{proof}
For $\mu\in\Lambda_r$, the $\GG$-module $\Sym^\mu V$ has highest weight $\sum_{i=1}^r \lambda_i\delta_i$, for $\lambda:=\mu^t$. Since 
$\Sym^\mu V$ is a direct summand of $V^{\otimes |\mu|}$, and hence a submodule of $V^{\otimes r}$, we find that $L_{\GG}(\xi)$ belongs to $\Rep^{(r)}\GG$. 
\end{proof}

\begin{proof}[Proof of Theorem~\ref{ThmOSp}] We freely use the observations in~\ref{defCM}.
Since $V^{\otimes r}$ is a submodule of $T^r$, we have $\cO^{V^{\otimes r}}\subset\cO^{T^r}$. Since $T^r$ is in turn a submodule of a direct sum of copies of $V^{\otimes r}$, the latter inclusion is actually an equality.
Part (i) thus follows from Lemma~\ref{LemOOSp}.

Now we claim that the abelian subcategory scom-$\cO^{V^{\otimes r}}$ of $\Rep\GG$ is the category of all subquotients of direct sums of copies of $V^{\otimes r}$ and $\Pi V^{\otimes r}$. By definition, all such subquotients belong to the subcategory. For arbitrary $\{i_1,\cdots, i_r\}\subset \lbr1,m+2n\rbr^{\times r}$ it follows by direct computation that the morphism
$$\Pi^{\sum _a[i_a]}V^{\otimes r}\to \cO^{V^{\otimes r}},\quad e_{j_1}\otimes\cdots \otimes e_{j_r}\mapsto (-1)^{\sum_{a>b} [i_a][j_b]} X_{i_1 j_1}\cdots X_{i_r j_r}$$
is a comodule morphism in $\svec$. Hence $\cO^{V^{\otimes r}}$, as an object in $\Rep\GG$, is a quotient of a direct sum of copies of $V^{\otimes r}$ and $\Pi V^{\otimes r}$. Since every object in com-$\cO^{V^{\otimes r}}$ is a subobject of a direct sum of copies of $\cO^{V^{\otimes r}}$ and $\Pi \cO^{V^{\otimes r}}$, our claim follows.

Now we take $g\in \GG(\mk)$ as in \ref{DefOSp}. We can correspondingly choose $\cC_1$ as in Lemma~\ref{GenZ2} as the full subcategory of modules for which $\mZ_2=\{\varepsilon,g\}\subset \GG$ acts in the canonical way on the underlying $\mZ_2$-graded space, {\it i.e.} $a_M(v)=(-1)^{|v|}v$ for $v\in M$. Then the full subcategory of $\Rep\GG$ with objects belonging  both to scom-$\cO^{V^{\otimes r}}$ and $\cC_1$ is the one of all subquotients of direct sums of copies of $V^{\otimes r}$, by the claim in the above paragraph. By Lemma~\ref{GenZ2}, this category is equivalent to $\mbox{com-}\cO^{V^{\otimes r}}$.
By Lemma~\ref{LemOOSp} and \cite[\S 3.1]{Abe}, we have an equivalence
$$\cS^{o}_r(V)\mbox{-mod}\;\stackrel{\sim}{\to}\; \mbox{com-}\cO^{V^{\otimes r}},$$
which concludes the proof of part (ii).

By Lemma~\ref{LemTilt1} and Section~\ref{SecTilt} below, the algebra $\cS^o_r(V)$ has at most $|\Lambda_r|$ simple modules up to isomorphism, so Lemma~\ref{numbsimp} describes all simple modules. This proves part (iii).
\end{proof}

\subsection{The general linear case}
\label{SecGLSchur}

We fix $m,n\in\mN$, take $V\in\svec$ of $\dim V=(m|n)$ and set $W=V^\ast$ and $\delta=\sdim(V)$.

\subsubsection{} For $r,s\in\mN$, we set
$$T^{r,s}\;:=\;\bigoplus_{j=0}^{\min(r,s)}V^{\otimes (r-j)}\otimes W^{\otimes (s-j)}.$$
By \eqref{UnivG}, we have algebra morphisms 
$$\cB^{\cc}_{r,s}(\delta)\to \End_{\mk}(T^{r,s})\quad\mbox{and}\quad \cB_{r,s}(\delta)\to\End_{\mk}(V^{\otimes r}\otimes W^{\otimes s}).$$ We define the super algebra
$$\cS^g_{r,s}(V)\;:=\;\uEnd_{\cB_{r,s}}(V^{\otimes r}\otimes W^{\otimes s}).$$

\begin{thm}\label{ThmGL}
Set $\GG=\GL(V)$.
\begin{enumerate}[(i)]
\item We have $\cS^g_{r,s}(V)\simeq \uEnd_{\cB^{\cc}_{r,s}}(T^{r,s})$.
\item If $p\not=2$, the category $\cS^g_{r,s}(V)\mbox{{\rm -mod}}$ is equivalent to the abelian subcategory $\Rep^{(r,s)}\GG$ of modules in $\Rep\GG$ which are subquotients of direct sums of $V^{\otimes r}\otimes W^{\otimes s}$.
\item If $m\ge r$, $n\ge s$ and $p\not\in\lbr2,r\rbr$, the simple module $L_{\GG}(\xi)$ for $\xi\in X^+$ is contained in $\Rep^{(r,s)}\GG$ if and only if $\xi=\sum_{i=1}^r \lambda_i\epsilon_i-\sum_{j=1}^s\mu_{j}\delta_j$ for some $(\lambda,\mu)\in \Lambda_{r,s}$.
\end{enumerate}
\end{thm}

\begin{lemma}
With $\cO:=\cO[\GL(V)]$,
we have $\cS^g_{r,s}(V)\simeq (\cO^{V^{\otimes r} \otimes W^{\otimes s}})^\ast$  and $\uEnd_{\cB^{\cc}_{r,s}}(T^{r,s})\simeq (\cO^{T^{r,s}})^\ast$.\end{lemma}
\begin{proof}
These are applications of Lemma~\ref{LemNew}, by Theorem~\ref{ThmNew}.\end{proof}

\begin{proof}[Proof of Theorem~\ref{ThmGL}]
Mutatis mutandis the proof of Theorem~\ref{ThmOSp}.
\end{proof}

\begin{rem}
Some of these and other connections between walled Brauer algebras and general linear super groups appear for $\mk=\mC$ in \cite{BS} and \cite{EHS}.
\end{rem}

\subsection{The periplectic case}\label{SecPSchur}
Fix $n,r\in\mN$ and take $V\in \svec$ with odd form and $\dim V=(n|n)$.
\subsubsection{}  By \eqref{UnivP}, we have an algebra morphism
$ \cA_r\to\End_{\mk}(V^{\otimes r}).$ We define the super algebra
$$\cS_r^p(V)\;:=\;\uEnd_{\cA_r}(V^{\otimes r}).$$

\begin{thm}\label{ThmP} Set $\GG=\Pe(V)$ and $\cO=\cO[\GG]$.
\begin{enumerate}[(i)]
\item If $p\not=2$, the category $\cS^p_r(V)\mbox{{\rm -mod}}$ is equivalent to the abelian subcategory $\Rep^{(r)}\GG$ of modules in $\Rep\GG$ which are subquotients of direct sums of $V^{\otimes r}$.
\item We have $\cS_r^p(V)\simeq (\cO^{V^{\otimes r}})^\ast\simeq\uEnd_{\cA_r^{\cc}}(\oplus_{j\in \JJJ(r)}V^{\otimes j}).$
\end{enumerate}
\end{thm}
\begin{proof}
Mutatis mutandis the proof of Theorem~\ref{ThmOSp}.
\end{proof}

\begin{rem}
We do not study the category $\Rep^{(r)}\GG$ of Theorem~\ref{ThmP}(i) in the current paper. If $\mk=\mC$, a thorough study of $\Rep^{(r)}\GG$ has been made recently in \cite{EnS}, where it is shown in particular that the category is of highest weight type. This justifies Conjecture~\ref{Conj} below.
\end{rem}

\subsection{The queer case}
We fix $n,r\in\mN$ and consider $V\in \svec$ with an odd isomorphism and $\dim V=(n|n)$ and set $W=V^\ast$.
\subsubsection{}  By Section~\ref{SecOBC}, we have a super algebra morphism
$$ \cBC_{r,s}\to\uEnd_{\mk}(V^{\otimes r}\otimes W^{\otimes s}).$$ We define the super algebra
$$\cS_{r,s}^q(V)\;:=\;\uEnd_{\cBC_{r,s}}(V^{\otimes r}\otimes W^{\otimes s}).$$

\begin{thm}\label{ThmQ} Set $\GG=\mathsf{Q}(V)$ and $\cO=\cO[\GG]$.
\begin{enumerate}[(i)]
\item We have $\cS^q_{r,s}(V)\simeq (\cO^{V^{\otimes r}\otimes W^{\otimes s}})^\ast$.
\item The category of modules in $\svec$ of the super algebra $\cS^q_{r,s}(V)$ is equivalent to the abelian subcategory $\Rep^{(r,s)}\GG$ of representations in $\Rep\GG$ which are subquotients of direct sums of $V^{\otimes r}\otimes W^{\otimes s}$ and $\Pi(V^{\otimes r}\otimes W^{\otimes s})$.
\end{enumerate}
\end{thm}
\begin{proof}
Part (i) is an application of Lemma~\ref{LemNew}, by Theorem~\ref{ThmNew}.. Part (ii) follows as in the proof of Theorem~\ref{ThmOSp}, except that we do not apply Lemma~\ref{GenZ2} as we keep working with scom-$\cO^{M}$.
\end{proof}


\section{Ringel duality for Brauer algebras}\label{SecRingel}

\subsection{The Brauer algebra}
Fix $V\in\svec$ with an even form and set $(m|2n)=\dim V$ and $\delta=\sdim V$. As before, we set $T^r=\oplus_{j\in\JJJ(r)}V^{\otimes j}$. Whenever $T^r$ is considered as a $\cB^{\cc}_r(\delta)$-module its super structure is important in the definition of the action, but we consider it as an ordinary module of the non-super algebra $\cB^{\cc}_r(\delta)$.

By Proposition~\ref{PropCZ}(i) and (ii), the following theorem gives a description of the Ringel dual of the Brauer algebra, for all cases in which it is quasi-hereditary (ignoring trivial cases in which it is semisimple).
\begin{thm}\label{ThmTilt}
If $p\not\in\lbr2,r\rbr$ and $\min(m,n)\ge r$, then $T^r$ is a complete tilting module of $\cB^{\cc}_r(\delta)$.
A Ringel dual of $\cB^{\cc}_r(\delta)$ is thus given by $\cS^o_r(V)$.
\end{thm}
We will precede the proof with some lemmata and constructions.

\begin{lemma}\label{LemInjTilt}
If a $\cB^{\cc}_r(\delta)$-module is injective as an $\cR_r$-module and self-dual for the duality in Proposition~\ref{PropCZ}(v), it is a tilting module.
\end{lemma}
\begin{proof}
By Lemma~\ref{lemtilt}, a self-dual module $M$ is tilting if
$$\Ext^1_{\cB_r^{\cc}}(\Delta(\lambda),M)\;\simeq\; \Ext^1_{\cR_r}(S(\lambda), M)$$
vanish for all $\lambda\in\Lambda_r$. The isomorphism follows from Proposition~\ref{PropCZ}(iv) and (iii) and Shapiro's lemma. This concludes the proof.
\end{proof}

\subsubsection{} The tensor algebra of $V$ has an $\mN$-grading defined by placing $V_{\oa}$ in degree $0$ and $V_{\ob}$ in degree $1$. We have a corresponding grading of vector spaces,
$T^r=\oplus_{i=0}^{r} T^r[i].$
This allows to define a filtration of $T^r$ as an $\cR_r$-module
$$0=F_{-1}T^r\subset F_0 T^r\subset\cdots \subset F_rT^r=T^r,\quad\mbox{with}\quad F_j T^r\;=\;\bigoplus_{i=0}^{j} T^r[i]. $$
Note that $T^r$ has a canonical structure of an $\OG(m)\times \GL(2n)$-representation and each $T^r[i]$ is a direct summand of this representation. In particular $F_\bullet T^r$ is also a filtration of the $\OG(m)\times \GL(2n)$ representation $T^r$. We set $Q_j:= F_j T^r/F_{j-1}T^r$ and consider it as an $\cR_r$-module and $\OG(m)\times \GL(2n)$-representation.

\begin{lemma}\label{LemOGL}
For $0\le j\le r$, the image of 
$\cR_r\to \End_{\mk}(Q_j )$
is contained in $A:=\End_{\OG(m)\times \GL(2n)}(Q_j)$. If $\min(m,2n)\ge r$ this morphism $\cR_r\to A$ makes $A$ projective as a right $\cR_r$-module.
\end{lemma}
\begin{proof}
For simplicity, we use the canonical isomorphism of vector spaces $T^r[j]\stackrel{\sim}{\to}Q_j$.
Set $\GG=\OG(V_{\oa})\times \GL(V_{\ob})$. It suffices to show that a set of generators of $\cR_r$ is mapped to elements in $\End_{\GG}(T^r[j])$. We take a basis of $T^r[j]$ induced from a homogeneous basis of $V$. 
A diagram in $\cR_r$ consisting of one cap and otherwise only non-crossing propagating lines yields a morphism which is zero on basis elements unless the cap is evaluated on two elements of copies of $V_{\oa}$. In the latter case, we get the evaluation of the $\OG(V_{\oa})$-invariant bilinear form on $V_{\oa}$. Diagrams which belong to $\cH$ are mapped to braiding morphisms in $\Rep\GG$. The above two types of diagrams generate the algebra $\cR_r$.

If $\min(m,2n)\ge r$, the fundamental theorems of invariant theory, see e.g. \cite[Theorem~4.2 and Theorem~5.7]{Concini} imply that the algebra $A$ can be described diagrammatically as follows. 
Consider $i$ dots on a horizontal line and $k$ dots on a parallel line above the first one. On both lines there are $j$ white dots and the others are black. We define an ``$A$-diagram'' to be a graphical representation of a pairing of the dots such that each white dot is connected with a white dot on the other line. If $j=2$, an example of an $A$-diagram is
$$
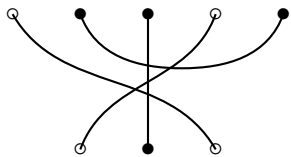
\begin{tikzpicture}[scale=0.9,thick,>=angle 90]
\begin{scope}[xshift=4cm]
\node at (0,0) {$\circ$};
\node at (1,0) {$\bullet$};
\node at (2,0) {$\circ$};
\node at (-1,2) {$\circ$};
\node at (0,2) {$\bullet$};
\node at (1,2) {$\bullet$};
\node at (2,2) {$\circ$};
\node at (3,2) {$\bullet$};
\draw (2,0) to [out=120, in=-60] +(- 3,2);
\draw (0,0) to [out=70, in=-110] +(2,2);
\draw  (1,0) -- +(0,2);
\draw (0,2) to [out=-70,in=180] +(1.5,-0.8) to [out=0,in=-110] +(1.5,0.8);
\end{scope}
\end{tikzpicture}
$$
The space $A$ has a basis consisting of all $A$-diagrams, with at most $r$ (and at least $j$) dots on each line and on each line a total number of dots in $\JJJ(r)$. The product of two diagrams is zero unless the dots on the lower line of the left diagram match the dots on the upper line of the right diagram. When the dots match, the product is given by concatenation, with evaluation of loops at $m$. We have an obvious interpretation of $A$-diagrams as morphisms. The above diagram is a morphism 
$$V_{\ob}\otimes V_{\oa}\otimes V_{\ob}\;\to\; V_{\ob}\otimes V_{\oa}^{\otimes 2}\otimes V_{\ob}\otimes V_{\oa}.$$
In particular, the morphism $\cR_r\to A$ is determined by the local relations
$$
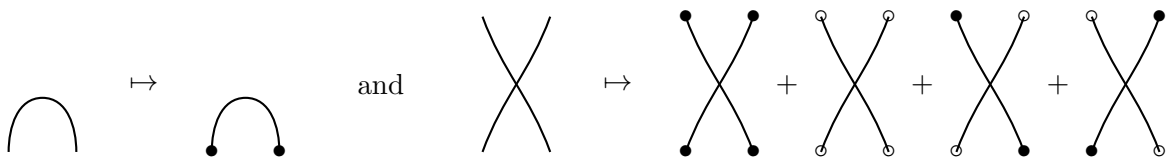
\begin{tikzpicture}[scale=0.9,thick,>=angle 90]
\begin{scope}[xshift=4cm]
\draw (1,0) to [out=90,in=-180] +(0.5,0.8) to [out=0,in=90] +(0.5,-0.8);
\node at (3,1) {$\mapsto$};
\draw (4,0) to [out=90,in=-180] +(0.5,0.8) to [out=0,in=90] +(0.5,-0.8);
\node at (4,0) {$\bullet$};
\node at (5,0) {$\bullet$};
\node at (6.5,1) {and};
\draw (8,0) to [out=70, in=-110] +(1,2);
\draw (9,0) to [out=110, in=-70] +(-1,2);
\node at (10,1) {$\mapsto$};
\draw (11,0) to [out=70, in=-110] +(1,2);
\draw (12,0) to [out=110, in=-70] +(-1,2);
\node at (11,2) {$\bullet$};
\node at (12,2) {$\bullet$};
\node at (11,0) {$\bullet$};
\node at (12,0) {$\bullet$};
\node at (12.5,1) {$+$};
\draw (13,0) to [out=70, in=-110] +(1,2);
\draw (14,0) to [out=110, in=-70] +(-1,2);
\node at (13,2) {$\circ$};
\node at (14,2) {$\circ$};
\node at (13,0) {$\circ$};
\node at (14,0) {$\circ$};
\node at (14.5,1) {$+$};
\draw (15,0) to [out=70, in=-110] +(1,2);
\draw (16,0) to [out=110, in=-70] +(-1,2);
\node at (15,2) {$\bullet$};
\node at (16,2) {$\circ$};
\node at (15,0) {$\circ$};
\node at (16,0) {$\bullet$};
\node at (16.5,1) {$+$};
\draw (17,0) to [out=70, in=-110] +(1,2);
\draw (18,0) to [out=110, in=-70] +(-1,2);
\node at (17,2) {$\circ$};
\node at (18,2) {$\bullet$};
\node at (17,0) {$\bullet$};
\node at (18,0) {$\circ$};
\end{scope}
\end{tikzpicture}
$$

Now we prove that $A_{\cR_r}$ is projective. First observe that, by definition, $A$ is Morita equivalent to its centraliser algebra $fAf$ with $f\in A$ the sum over all diagrams with only non-intersecting propagating lines with the $j$ white dots on the left. It thus suffices to prove that $fA_{\cR_r}$ is projective. It follows easily that $fA$ is isomorphic to a direct summand of $\cB_r^{\cc}(\delta)$ as a right $\cR_r$-module, with the isomorphism realised by forgetting the colour of the dots. This means the claim follows from Proposition~\ref{PropCZ}(iii). 
\end{proof}

\begin{lemma}\label{LemTilt1}
If $\min(m,2n)\ge r$ and $p\not\in\lbr2,r\rbr$, then $T^r$ is a tilting module for $\cB^{\cc}_r(\delta)$.
\end{lemma}
\begin{proof}
It follows from a straightforward computation that the canonical isomorphism from $T^r$ to $(T^r)^\ast$, induced by $V\to V^\ast$ with $v\mapsto \langle v,\cdot\rangle$, makes $T^r$ self-dual with respect to the duality in Proposition~\ref{PropCZ}(v).
By Lemma~\ref{LemInjTilt} it thus suffices to prove that $T^r$ is injective as a $\cR_r$-module.
We prove the stronger statement that each $\cR_r$-module $Q_j$ is injective.
Under the assumptions on $p$, the algebras $A$ in Lemma~\ref{LemOGL} are semisimple, see e.g. \cite{Rui}.
Consequently, by Lemma~\ref{LemOGL}, the functor
$$\Hom_{\cR_r}(-, Q_j)\;\simeq\; \Hom_{A}(A\otimes_{\cR_r}-,Q_j)$$
is exact and hence $Q_j$ is injective as a $\cR_r$-module.
\end{proof}
\begin{proof}[Proof of Theorem~\ref{ThmTilt}]
By Lemma~\ref{LemTilt1}, the $\cB^{\cc}_r(\delta)$-module $T^r$ is tilting. By Lemma~\ref{numbsimp},
$\End_{\cB^{\cc}_r}(T^r)$ has as many simple modules as $\cB^{\cc}_r(\delta)$, which implies that $T^r$ is complete.
\end{proof}

\subsection{The walled Brauer algebra}
Resume the notation of Subsection~\ref{SecGLSchur}.
By equipping the tensor algebra of $V\oplus W$ with an $\mN$-grading with degree of $V_{\oa}\oplus W_{\oa}$ equal to $0$ and degree of $V_{\ob}\oplus W_{\ob}$ equal to $1$, we get a filtration of $T^{r,s}$ as a $\GL(V_{\oa})\times \GL(V_{\ob})\times \GL(W_{\ob})$-representation.
By adapting the proof of Theorem~\ref{ThmTilt} appropriately, we find the following result.
\begin{thm}\label{ThmTiltW}
Assume $p\not\in\lbr2,\max(r,s)\rbr$, $m\ge r+s$, $n\ge \max(r,s)$ and $m-n=\delta$. Then $T^{r,s}$ is a complete tilting module of $\cB^{\cc}_{r,s}(\delta)$.
A Ringel dual of $\cB^{\cc}_{r,s}(\delta)$ is thus given by $\cS^g_{r,s}(V)$.
\end{thm}

\subsection{The periplectic Brauer algebra}

It was proved in \cite{PB1} that $\cA_r^{\cc}$ is quasi-hereditary when $p\not\in\lbr2,r\rbr$ and Morita equivalent to $\cA_r$ when $r$ is odd. Resume the notation of Subsection~\ref{SecPSchur}.
\begin{conj}\label{Conj}
Take $n\ge 2 r $ and assume $p\not\in\lbr2,r\rbr$. A Ringel dual of $\cA_r^{\cc}$ is given by $\cS^p_r(V)$ for $V\in\svec$ with $\dim V=(n|n)$ equipped with an odd form.
\end{conj}
In characteristic zero there there is a proof of the conjecture, which goes in the opposite direction from our above proof for the (walled) Brauer algebra, by using Proposition~\ref{PropRingel}(ii), based on recent results in~\cite{EnS}.
\begin{prop}[Entova-Eizenbud, Serganova]
Conjecture \ref{Conj} is true for $\mk=\mC$ and $n\ge 8r$.
\end{prop}
\begin{proof}
In \cite[Propositions~5.2.1 and~9.2.3]{EnS} it is proved that $\Rep^{(r)}\Pe(V)$ is a highest weight category where the tilting modules are precisely the direct summands of 
$$T^r\,=\,\bigoplus_{j\in \JJJ(r)}\Pi^{\frac{r-j}{2}} (V^{\otimes j}).$$ The Ringel dual of $\Rep^{(r)}\Pe(V)$ is thus $\End_{\mathfrak{pe}(V)}(T^r)$. The latter algebra is precisely $\cA^{\cc}_r$, see e.g. \cite[Proposition~8.3.3(i) and Theorem~7.3.1(ii)]{PB1} or \cite[Theorem~5.4.2(ii)]{Kujawa}.
\end{proof}


\section{Applications to super groups and Lie algebras}\label{SecApp}
We list some consequences of our results on Brauer algebras for the representation theory of super groups. Some results are extensions of known results from characteristic zero to positive characteristic, but with entirely different proofs.

\subsection{The orthosymplectic super group}
Resume the notation of Subsection~\ref{SchurO}. In particular $V$ has dimension $(m|2n)$ and is equipped with an even form $\langle\cdot,\cdot\rangle$.
\begin{thm}
Assume $p\not\in\lbr2,r\rbr$ and $\min(m,n)\ge r$. Then $\Rep^{(r)}\OSp(V)$ is a highest weight category for partial order $\xi<\eta$ if and only if $|\xi|< |\eta|$ with $|\xi|=|\lambda|$ for $\xi=\sum_{i=1}^n\lambda_i\delta_i$, with $\lambda\in\Lambda_r$. Furthermore, $\Rep^{(r)}\OSp(V)$ does not depend on $V$, up to equivalence.
\end{thm}
\begin{proof}
By Theorems~\ref{ThmOSp}(ii), the statements are about $\cS^o_r(V)$-mod. By Theorem~\ref{ThmTilt}, the algebra $\cS^o_r(V)$ is Ringel dual to $\cB_r^{\cc}(\delta)^{\op}\simeq \cB_r^{\cc}(\delta)$. 
This results thus follow from Propositions~\ref{PropRingel}(i) and~\ref{PropCZ}(iv).
\end{proof}

\begin{thm}\label{ConseqFT}
Assume $p\not\in\lbr2,r\rbr$ and $\min(m,n)\ge r$ and set $\delta=m-2n$. 
\begin{enumerate}[(i)]
\item We have an algebra isomorphism $$\cB_r(\delta)\;\stackrel{\sim}{\to}\; \uEnd_{\OSp(V)}(V^{\otimes r}).$$
\item The space of $\OSp(V)$-invariant multilinear forms $V^{\times 2r}\to \mk$ is spanned by all transforms under the symmetric group of 
$$v_1,v_2,\cdots,v_{2r}\;\mapsto\; \langle v_1, v_{2}\rangle\cdots\langle v_{2r-1},v_{2r}\rangle.$$
\end{enumerate}
\end{thm}
\begin{proof}
By Proposition~\ref{PropRingel}(iii) and Theorem~\ref{ThmTilt} we have algebra isomorphisms
$$\cB_r^{\cc}(\delta)\;\stackrel{\sim}{\to}\;\uEnd_{\cS^o_r(V)}(T^{ r})  \;\stackrel{\sim}{\to}\;\uEnd_{\OSp(V)}(T^{ r}).$$
Part (i) is a restriction of this isomorphism. Part (ii) then follows from applying the isomorphism between $\uEnd(V^{\otimes r})$ and $\uHom(V^{\otimes 2r},\mk)$, see e.g.~\cite{BrCat, Kujawa}.
\end{proof}

\begin{rem}
\begin{enumerate}[(i)]
\item For $\mk=\mC$, it was recently proved in \cite{Yang, Sel} that the isomorphism in Theorem~\ref{ConseqFT}(i) holds more generally if $r<(m+1)(n+1)$. In \cite[Theorem~A]{ES}, the case $\mk=\mC$ and $r\le m+n$ was proved. Our result in positive characteristic seems to be new.
\item For $\mk=\mC$, but without the condition $\min(m,n)\ge r$, Theorem~\ref{ConseqFT}(ii) is the main result of \cite{DLZ}, proved through geometric methods.
\end{enumerate}
\end{rem}

\subsection{The general linear super group}

\begin{thm}\label{CorTiltW}
Assume $p\not\in\lbr2,\max(r,s)\rbr$, $m\ge r+s$ and $n\ge \max(r,s)$. Then $\Rep^{(r,s)}\GL(V)$ is a highest weight category which does not depend on $V$.
\end{thm}

\begin{proof}
This is an application of Proposition~\ref{PropRingel}(i) by Theorem~\ref{ThmTiltW}.
\end{proof}

\begin{rem}
If $\mk=\mC$ and $4(r+s)\le \min(m,n)$, the results in Theorem~\ref{CorTiltW} were first proved in \cite[Theorem~7.1.1 and~Corollary~8.5.2]{EHS} through completely different methods.
\end{rem}

\begin{prop}
Assume $p\not\in\lbr2,r\rbr$ and $\min(m,n)\ge r$, then we have an isomorphism $\mk\SG_r\stackrel{\sim}{\to}\uEnd_{\GL(V)}(V^{\otimes r})$.
\end{prop}
\begin{proof}
Mutatis mutandis Theorem~\ref{ConseqFT}(i).
\end{proof}

\begin{prop}\label{CorGL}
If $\mk=\mC$, the morphism $U(\mathfrak{gl}(V))\to \uEnd_{\cB_{r,s}}(V^{\otimes r}\otimes W^{\otimes s})$ is surjective.
\end{prop}
\begin{proof}
It follows from Lemma~\ref{LemDist}(ii) that the morphism $\Dist(\cO)\to X^\ast$ is surjective for any finite dimensional space $X\subset \cO$. Combined with Lemma~\ref{LemDist}(i) this shows that $U(\mathfrak{gl}(V))\to (\cO^{V^{\otimes r}\otimes W^{\otimes s}})^\ast$ is surjective. The conclusion follows from Theorem~\ref{ThmGL}(i) and diagram~\eqref{CommD}.
\end{proof}

\subsection{The periplectic super group}

The following answers in particular \cite[Question~8.1.6]{PB1}.
\begin{prop}\label{CorP}
If $\mk=\mC$, the morphism $U(\mathfrak{pe}(V))\to \uEnd_{\cA_r}(V^{\otimes r})$ is surjective.
\end{prop}
\begin{proof}
Mutatis mutandis Proposition~\ref{CorGL}.\end{proof}

\subsection{The queer super group}

\begin{prop}
If $\mk=\mC$, the morphism $U(\mathfrak{q}(V))\to \uEnd_{\cBC_{r,s}}(V^{\otimes r}\otimes W^{\otimes s})$ is surjective.
\end{prop}
\begin{proof}
Mutatis mutandis Proposition~\ref{CorGL}.\end{proof}


\appendix

\section{Ringel duality}\label{AppRingel}
In this section we recall some results from \cite{CPS, Ringel}. We work in $\vvec$, not in $\svec$ or $\VVec$.

\subsection{Definitions}

Consider a (finite dimensional unital associative) algebra~$A$ over $\mk$. 
We label the isoclasses of simple left $A$-modules by the (finite) set~$\Lambda=\Lambda_A$. When we consider a partial order $\le$ on~$\Lambda_A$, we will write $(A,\le)$. 
The simple modules are denoted by~$\{L(\lambda)\,|\,\lambda\in\Lambda\}$ and their respective projective covers and injective envelopes in $A$-mod by~$P(\lambda)$ and $I(\lambda)$.

\begin{ddef}\label{DefQH}
An algebra~$(A,\le)$ is {\bf quasi-hereditary} if there exist~$\{\Delta(\lambda)\,|\, \lambda\in \Lambda\}$ in $A$-mod such that, for all $\lambda,\mu\in\Lambda$:
\begin{enumerate}[(a)]
\item $[\Delta(\lambda):L(\lambda)]=1$;
\item $[\Delta(\lambda):L(\mu)]=0$ unless~$\mu\le\lambda$;
\item there is an epimorphism $P(\lambda)\tto \Delta(\lambda)$ where the kernel has a filtration with each quotient of the form $\Delta(\nu)$, with $\nu> \lambda$.
\end{enumerate}
\end{ddef}
Quasi-heredity is thus a property of the Morita class of an algebra. The conditions in the above definition actually state that $A$-mod is a {\bf highest weight category}. We now fix a quasi-hereditary algebra $(A,\le)$ for the remainder of this appendix.

\subsubsection{}\label{DefGood}It follows that for each $\lambda\in\Lambda$, the standard module $\Delta(\lambda)$ is the maximal quotient of $P(\lambda)$ for which $[\Delta(\lambda):L(\mu)]=0$ unless~$\mu\le\lambda$. We define dually the costandard module $\nabla(\lambda)$ as the maximal submodule of $I(\lambda)$ for which $[\nabla(\lambda):L(\mu)]=0$ unless~$\mu\le\lambda$.

Assume that $A$ has an anti-automorphism $\varphi$. We have the corresponding contragredient equivalence
$$\dd:\; A\mbox{-mod}\; \stackrel{\sim}{\to}\; A\mbox{-mod}, \quad M\mapsto (M^\ast)_\varphi.$$
This induces a bijection $\lambda\mapsto \lambda^\varphi$ on $\Lambda$ such that $\dd L(\lambda)\simeq L(\lambda^\varphi)$. If $(\cdot)^\varphi$ induces an automorphism of the poset $(\Lambda,\le)$, we refer to $\dd$ as a {\bf good duality}. In particular, we then have $\dd\Delta(\lambda)\simeq \nabla(\lambda^\varphi)$.

\subsubsection{}\label{SecTilt} 
A filtration of a module where each quotient is a standard, resp. costandard, module is referred to as a $\Delta$-flag, resp. $\nabla$-flag.
We refer to modules which admit a both a $\Delta$-flag and a $\nabla$-flag as {\bf tilting} modules. By \cite[Section~5]{Ringel}, every direct summand of a tilting module is again a tilting module and there exist precisely $|\Lambda|$ indecomposable tilting modules up to isomorphism. We label them as $\{T(\lambda)\,|\,\lambda\in\Lambda\}$, where $T(\lambda)$ has $\Delta(\lambda)$ as a submodule for which the quotient has a $\Delta$-flag. We say that a tilting module is {\bf complete} if each $T(\lambda)$ appears as a direct summand.

An algebra of the form
$\End_A(T)^{\op},$
for a complete tilting module $T$ of $A$, is a {\bf Ringel dual} of $A$. All Ringel duals of $A$ are clearly Morita equivalent.

\subsection{Properties}

The following is an immediate consequence of \cite[Theorem~4]{Ringel}.
\begin{lemma}\label{lemtilt}
Let $(A,\le)$ be a quasi-hereditary algebra with good duality $\dd$. If an $A$-module $M$ satisfies $M\simeq \dd M$ and
$$\Ext^1_A(\Delta(\lambda),M)=0,\qquad\mbox{for all $\lambda\in\Lambda$},$$
then it is a tilting module.
\end{lemma}

\begin{prop}[Ringel]\label{PropRingel}
Let $(A,\le)$ be quasi-hereditary with complete tilting module $T$ and set $R(A):=\End_{A}(T)$. 
\begin{enumerate}[(i)]
\item We have a bijection $\Lambda_A\stackrel{\sim}{\to}\Lambda_{R(A)^{\op}}$, denoted by $\lambda\mapsto \widetilde{\lambda}$, where $P(\widetilde{\lambda})\simeq\Hom_A(T,T(\lambda))$. Then $(R(A)^{\op},\le)$ is quasi-hereditary, with $\widetilde{\lambda}\le \widetilde{\mu}$ if and only if $\mu\le\lambda$.
\item The original $(A,\le)$ is a Ringel dual of $(R(A)^{\op},\le)$.
\item The morphism $A\to\End_{\mk}(T)$ yields an isomorphism
$A\stackrel{\sim}{\to}\End_{R(A)}(T).$
\end{enumerate}
\end{prop}
\begin{proof}
Part (i) follows from \cite[Theorem~6]{Ringel}.
Part (ii) follows from \cite[Theorem~7]{Ringel}.
Since $T$ is a faithful $A$-module, see \cite[Theorem~5]{Ringel}, to prove part (iii) it suffices to show that $\End_{R(A)}(T)$ has the same dimension as $A$. The latter follows as in the proof of \cite[Theorem~7]{Ringel}.\end{proof}

\subsection*{Acknowledgement}
The research was supported by ARC grants DE170100623 and DP180102563. The author thanks Sigiswald Barbier for useful comments on the first version of the manuscript.

\begin{flushleft}
	K. Coulembier\qquad \url{kevin.coulembier@sydney.edu.au}
	
	School of Mathematics and Statistics, University of Sydney, NSW 2006, Australia
	
	\end{flushleft}

\end{document}